\documentclass[a4paper,12pt]{amsart}
\usepackage{amsmath, amssymb, amsthm, xcolor}
\usepackage{verbatim}
\usepackage{xypic}
\usepackage{graphicx}
\usepackage{enumerate}
\usepackage[colorlinks=true, citecolor=blue, linkcolor=blue, bookmarks=true]{hyperref}
\usepackage{tikz}
\usetikzlibrary{intersections}
\usepackage{addlines}

\newcounter{dummy}

\newtheorem{theorem}[dummy]{Theorem}

\newtheorem{corollary}[dummy]{Corollary}
\newtheorem{proposition}[dummy]{Proposition}

\newtheorem*{theorem*}{Theorem}

\theoremstyle{definition}
\newtheorem{definition}[dummy]{Definition}

\newcommand{\norm}[1]{\left \lVert #1 \right \rVert}

\DeclareMathOperator{\id}{id}

\newcommand{\N}{\mathbb{N}}

\newcommand{\K}{\mathbb{K}}
\renewcommand{\O}{\mathcal{O}}

\allowdisplaybreaks

\author[Easo et al.]{Philip Easo}
\author[]{Esperanza Garijo}
\author[]{Sarunas Kaubrys}
\author[]{David Nkansah}
\author[]{Martin Vrabec}
\author[]{David Watt}
\author[]{Cameron Wilson}
\author[]{Christian B{\"o}nicke}
\author[]{Samuel Evington}
\author[]{Marzieh Forough}
\author[]{Sergio Gir{\'o}n Pacheco}
\author[]{Nicholas Seaton }
\author[]{Stuart White}
\author[]{Michael F. Whittaker}
\author[]{Joachim Zacharias}

\thanks{Research partially supported by EPSRC grant EP/R025061/1.  CB was partially supported by the Alexander von Humboldt-Foundation. MF was supported by GA\v{C}R project 19-05271Y,  RVO: 67985840, and a scheme 2 grant of the London Mathematical Society. MW was partially supported by EPSRC grant EP/R013691/1.}

\begin{document}
	\title{The Cuntz--Toeplitz algebras have nuclear dimension one}

\maketitle
\begin{abstract}
  We prove that unital extensions of Kirchberg algebras by separable stable AF algebras have nuclear dimension one. The title follows.
\end{abstract}
\section{Introduction}

A compact space $X$ has covering dimension at most $n$ if every open cover can be refined and then $(n+1)$-coloured so that sets with the same colour are disjoint (\cite{Os65}). Nuclear dimension is a generalisation of covering dimension to $C^*$-algebras, introduced by Winter and the fifteenth-named author (\cite{WZ10}), based on coloured finite-dimensional approximations, and has played a prominent role in the classification programme for amenable $C^*$-algebras over the last decade (see the expository articles \cite{ET08,Wi18}). 

For simple $C^*$-algebras, the theory of the nuclear dimension, and its relationship with classification is essentially complete.  The only possible values of the nuclear dimension are $0$, $1$ and $\infty$ and these values can be seen in the internal structure of the algebra (\cite{WZ10,MS12,MS14,SWW15,BBSTWW,CETWW,CE}). Moreover, (in the presence of the universal coefficient theorem), the simple separable unital and nuclear $C^*$-algebras accessible to classification by $K$-theory and traces are precisely those of nuclear dimension $0$ and $1$ (\cite{Kir95,Phi00,Go15, EGLN15, TWW17}).  Accordingly, attention is now turning to the non-simple setting. How do we compute the nuclear dimension of non-simple $C^*$-algebras, and is there a connection between finiteness of the nuclear dimension and classification?

In this paper, we consider $C^*$-algebras arising from extensions $0\rightarrow J \rightarrow A \rightarrow B\rightarrow 0$. A key advantage of nuclear dimension over its predecessor (the decomposition rank, defined in \cite{KW04}) is its compatibility with extensions: one can directly obtain an $(n+m)$-coloured approximation for $A$ by combining an $n$-coloured approximation for $J$ with an $m$-coloured approximation for $B$ (\cite[Proposition 2.9]{WZ10}).  If the extension is a direct sum (and more generally for quasidiagonal extensions), one can do better because $B$ and $J$ can be approximated in an orthogonal fashion (\cite{KW04, WZ10, RSS15}).  This results in an approximation for $A$ with only $\max(n,m)$ colours, which is also the fewest colours possible. For general extensions, we do not know where in the range between $\max(n,m)$ and $n+m$ the correct number of colours for $A$ falls.

For a decade, the nuclear dimension of the Toeplitz algebra $\mathcal{T}$ was an open problem because the extension $0 \rightarrow\mathbb{K} \rightarrow \mathcal{T} \rightarrow C(S^1)\rightarrow 0$ is not quasidiagonal and moreover $\mathcal{T}$ is the primordial example of an extension with infinite decomposition rank, despite the ideal and quotient having finite decomposition rank.  It was known that two colours were necessary in a nuclear dimension approximation of $\mathcal T$ and three suffice because two colours are needed to approximate $C(S^1)$ and one for $\mathbb{K}$ (\cite{WZ10}). Recently, Brake and Winter found an elegant way to re-use one colour to get a 2-coloured approximation to $\mathcal{T}$, thereby proving that the Toeplitz algebra has nuclear dimension one (\cite{BW19}).

Motivated by the work of Brake and Winter, we achieve a similar colour reduction for a large class of unital extensions of Kirchberg algebras (simple, separable, amenable, purely infinite $C^*$-algebras).

\begin{theorem}\label{MainTheorem}
	Let $0\rightarrow J \rightarrow A \rightarrow B\rightarrow 0$ be a unital extension where $J$ is a separable stable AF algebra and $B$ is a Kirchberg algebra. Then $A$ has nuclear dimension one. 
\end{theorem} 

The class of $C^*$-algebras covered by Theorem \ref{MainTheorem} includes the Cuntz--Toeplitz algebras $\mathcal{T}_n$ (for $n \geq 2$), defined as the universal $C^*$-algebras generated by isometries $s_1,\ldots,s_n$ with pairwise orthogonal ranges. Indeed, the ideal of $\mathcal{T}_n$ generated by $1-\sum_{i=1}^n s_is_i^*$ is isomorphic to the compacts $\mathbb{K}$ and the quotient by this ideal is the Cuntz algebra $\mathcal{O}_n$ (\cite{Cu77}). The Toeplitz algebra can be viewed as the Cuntz--Toeplitz algebra $\mathcal{T}_1$, but in this case the quotient is the commutative $C^*$-algebra $C(S^1)$ rather than a Kirchberg algebra. It is therefore not surprising that our proof requires a different set of tools from Brake and Winter while sharing some structural similarities.

Although not needed in our proof, Gabe and Ruiz have shown that the $C^*$-algebras covered by Theorem \ref{MainTheorem} (and that satisfy the universal coefficient theorem) are classified by their 6-term exact sequences in $K$-theory together with order and the position of the unit in the $K_0$-groups (\cite{GR20}). Together with our result, this is further evidence that the relationship between nuclear dimension at most one and $K$-theoretic classification persists outside the simple setting. This connection has long been evident for AF algebras, which have nuclear dimension zero and are classified by $K_0$ (\cite{El76}), and is also present for the class of non-simple algebras classified in \cite{El93}, which are built as limits of suitable one-dimensional building blocks and so have nuclear dimension at most $1$. Furthermore, amenable $\mathcal{O}_\infty$-stable algebras have recently been shown to have nuclear dimension one (\cite{BGSW19}). This last class was classified by ideal related $KK$-theory by Kirchberg (\cite{Ki00}; see \cite{Ga19} for a new approach to this result).

\medskip

We now discuss the proof of our theorem. The basic approximation has the form
\begin{align}\label{Intro.E1}
	a &\approx h_i^{1/2}ah_i^{1/2} + (h_i-h_{i+1})^{1/2}\mu(\pi(a))(h_i-h_{i+1})^{1/2}\nonumber\\&\quad + (1-h_{i+1})^{1/2}\mu(\pi(a))(1-h_{i+1})^{1/2},
\end{align} 
where $\pi:A \rightarrow B$ is the quotient map, $\mu:B \rightarrow A$ is a c.p.c.\ splitting and  $(h_i)_i$ is a quasicentral approximate unit for $J$ relative to $A$ such that $h_{i+1}h_i = h_i$.
Note that the first term on the right of (\ref{Intro.E1}) is in the ideal, the last term factors through the quotient, and the first and last terms are orthogonal by the hypothesis on $(h_i)_i$. If the middle term were to vanish, the extension would be quasidiagonal, and the nuclear dimension of $A$ would just be the maximum of that of the ideal and the quotient. 

In \cite{BW19}, Brake and Winter handle the non-vanishing middle term in the case of the Toeplitz algebra by using Lin's theorem on almost normal elements of matrix algebras (\cite{Li97}) to prove a lifting theorem for order zero maps from $C(S^1)$ into ultraproducts of finite-dimensional $C^*$-algebras. A similar strategy is not available when $B$ is simple and infinite as there are no non-trivial order zero maps from $B$ into finite-dimensional $C^*$-algebras. 

The new idea in our proof is inspired by a technique that has been utilised in the study of $\mathcal{Z}$-stable and $\mathcal{O}_\infty$-stable algebras (\cite{BBSTWW, CETWW, BGSW19}). Breaking the last term in half, for $b\in B$, we set
\begin{align}
\theta^{(0)}_i(b) &= (h_i-h_{i+1})^{1/2}\mu(b)(h_i-h_{i+1})^{1/2}\nonumber\\
&\quad + (1-h_{i+1})^{1/2}\mu(k_i^{1/2}bk_i^{1/2})(1-h_{i+1})^{1/2},\nonumber\\
	\theta^{(1)}_i(b) &= (1-h_{i+1})^{1/2}\mu((1-k_i)^{1/2}b(1-k_i)^{1/2})(1-h_{i+1})^{1/2}, 
\end{align} 
where $(k_i)_i$ is an approximately central sequence of positive contractions in $B$ and each $k_i$ has spectrum $[0,1]$. We then construct 1-coloured approximations for $\theta^{(0)}_i$ and $\theta^{(1)}_i$, by building models which conspicuously have 1-coloured approximations using Voiculescu's quasidiagonality of the cone $C_0(0,1] \otimes B$, and then using classification results to compare these models with $\theta^{(0)}_i$ and $\theta^{(1)}_i$. Since $a \approx h_i^{1/2}ah^{1/2}_i + \theta^{(0)}_i(\pi(a)) + \theta^{(1)}_i(\pi(a))$ and $\theta^{(1)}_i$ is orthogonal to $h_i$, we can re-use one colour and get a 2-coloured approximation for $A$.\footnote{We caution the reader comparing the outline above to our main proof, that in fact we first fix $(k_i)_i$, and subsequently obtain the quasicentral approximate unit $(h_i)_i$ in terms of it.} The classification results we need are in the non-simple setting, in the spirit of Kirchberg's classification of strongly purely infinite nuclear $C^*$-algebras (\cite{Kir95, Ki00}). Precisely, we need a classification theorem for full $^*$-homomorphisms from cones over Kirchberg algebras into ultrapowers, which we deduce from Gabe's classification theorem for full, nuclear, strongly $\mathcal{O}_\infty$-stable $^*$-homomorphisms (\cite{Ga19}). We record this result as Proposition \ref{thm:OinfClassification}. The role of the positive contractions $k_i$ with spectrum $[0,1]$ is to ensure that the $^*$-homomorphisms $C_0(0,1] \otimes B \rightarrow A_\omega$ coming from the approximately order zero maps $\theta^{(0)}_i$ and $\theta^{(1)}_i$ are full, so that Gabe's theorem applies. A criterion for checking fullness of such maps is established in Proposition \ref{prop:Fullness}, and the construction of the auxiliary maps using quasidiagonality of cones and order zero lifting is isolated in Proposition \ref{thm:Existence}. With the more technical aspects at hand, we can then give a streamlined proof of the main result at the end of the paper.

\section*{Acknowledgements}  

This paper is the result of an undergraduate summer research project at the University of Glasgow, where the first seven named authors were mentored by the last eight authors. PE, DN, MV and DW where supported by EPSRC EP/R025061/1, and EG, SK and CW by summer project awards from the School of Mathematics and Statistics at the University of Glasgow. The authors would also like to thank Jamie Antoun, Dimitrios Gerontogiannis, Alex Kumjian, Aidan Sims, Hannes Thiel and Jamie Walton for contributing to the success of the summer project, and Jamie Gabe for his very helpful comments on the first draft of this paper.

\section{Nuclear Dimension of extensions of Kirchberg algebras}

We begin with some preliminaries. Two positive elements $a$ and $b$ in a $C^*$-algebra $A$ are called orthogonal if $ab=0$. A completely positive (c.p.) map $\varphi:A\rightarrow B$ between two $C^*$-algebras is said to be \textit{order zero} if it sends orthogonal elements to orthogonal elements. The one-to-one correspondence between completely positive and contractive (c.p.c.) order zero maps $A\rightarrow B$ and $^\ast$-homomorphisms $C_0(0,1]\otimes A\rightarrow B$ established in \cite[Corollary~3.1]{WZ09} will be used repeatedly in this paper. The $^*$-homomorphism $\Phi:C_0(0,1]\otimes A\rightarrow B$ associated to a c.p.c.\  order zero map $\phi:A\rightarrow B$ is characterised by
\begin{equation}\Phi(\id_{(0,1]}\otimes\ a)=\phi(a),\quad a\in A.\end{equation}

We now recall the definition of nuclear dimension of a $C^*$-algebra and more generally the nuclear dimension of a $^*$-homomorphism.

\begin{definition}[{c.f.\ \cite[Definition 2.1]{WZ10}}]\label{def:DimNuc}
A $^*$-homomorphism of $C^*$-algebras $\pi:A \to B$ is said to have \emph{nuclear dimension} at most $n$, denoted $\dim_{\mathrm{nuc}}(\pi) \leq n$, if for any finite subset $\mathcal{F} \subseteq A$ and any $\varepsilon > 0$, there exists a finite dimensional $C^*$-algebra $F$ and completely positive maps
\begin{equation} \psi: A \rightarrow F \text { and } \phi: F \rightarrow B \end{equation}
such that $\psi$ is contractive and $\phi$ is $n$-decomposable, in the sense that we can write
\begin{equation} F=F^{(0)} \oplus \ldots \oplus F^{(n)} \end{equation}
so that $\phi|_{F^{(i)}}$ is c.p.c.\  order zero for all $i$, and for every $a \in \mathcal{F}$,
\begin{equation} \|\pi(a) - \phi\psi(a)\| < \varepsilon. \end{equation}
The \emph{nuclear dimension} of a $C^*$-algebra $A$, denoted $\dim_{\mathrm{nuc}}(A)$, is then defined as the nuclear dimension of the identity homomorphism $\id_A$.
\end{definition}

We will use ultraproducts of C$^*$-algebras throughout the paper. Recall that a \emph{filter} on $\N$ is a non-empty collection of non-empty subsets of $\N$ that is closed under taking supersets and finite intersections. Maximal filters are called \emph{ultrafilters}, and an ultrafilter is said to be \emph{free} if it contains all cofinite subsets of $\N$; free ultrafilters exist as a consequence of the axiom of choice.  We write $\lim_{i\rightarrow \omega} x_i$ for the limit along the ultrafilter $\omega$ of the bounded sequence of real numbers $(x_i)_i$.\footnote{By definition, $L = \lim_{i\rightarrow \omega} x_i$ means  $\forall\epsilon > 0, \; \{i \in \N: |x_i - L| < \epsilon\} \in \omega$. For bounded sequences, the limit along an ultrafilter always exists (and is unique).} 

Fix a free ultrafilter $\omega$ on $\mathbb{N}$ for the entirety of this paper. 
Let $(A_i)_i$ be a sequence of $C^*$-algebras. The \emph{bounded sequence algebra} is the $C^*$-algebra
\begin{equation}
    \prod_{i=1}^\infty A_i:=\lbrace (a_i)_i: a_i\in A_i\ \mathrm{ and }\ \sup_{i\in \mathbb{N}}\norm{a_i}<\infty\rbrace,
\end{equation}
and the \emph{ultraproduct} is the quotient $C^*$-algebra
\begin{equation}
   \prod_{i\rightarrow \omega} A_i := \frac{\prod_{i=1}^\infty A_i}{\lbrace (a_i)_i: \lim_{i\rightarrow \omega}\norm{a_i}=0\rbrace}.
\end{equation}
When $(A_i)_i$ is a constant sequence, i.e.\ $A_i=A$ for all $i$, we denote the ultraproduct by $A_\omega$, and there is a canonical embedding $A\rightarrow A_\omega$ via constant sequences.
As is standard, when working with ultraproducts we usually use representative sequences in the bounded sequence algebra to denote elements in the ultraproduct; this should be understood as standing for the image under the quotient.

We write $\sigma(a)$ for the spectrum of the element $a$. The following standard proposition will be used later to calculate the spectrum of elements in ultraproducts.

\begin{proposition}\label{remark}
Let $(A_i)_i$ be a sequence of $C^*$-algebras and $k=(k_i)_i\in \prod_{i\rightarrow \omega}A_i$. For $\lambda \in \mathbb{C}$, if $\lbrace i \in \mathbb{N}: \lambda \in \sigma(k_i)\rbrace \in \omega$, then $\lambda \in \sigma(k)$.
\end{proposition}
\begin{proof}
Without loss of generality, assume each $A_i$ is unital. We prove the contrapositive. Suppose $\lambda \not \in \sigma(k)$. Then $\lambda1-k$ is invertible, so there exists $d=(d_i)_i$ representing an element of $\prod_{i\rightarrow\omega}A_i$ such that $(\lambda 1-k)d=d(\lambda1-k)=1$ in $\prod_{i\rightarrow\omega}A_i$.  Then $\lim_{i\rightarrow\omega}\|(\lambda 1-k_i)d_i-1\|=0$ and $\lim_{i\rightarrow\omega}\|d_i(\lambda 1-k_i)-1\|=0$, so the set of those $i \in \mathbb{N}$ for which both $(\lambda1-k_i)d_i$ and $d_i(\lambda 1-k_i)$ are invertible lies in $\omega$. If both $(\lambda1-k_i)d_i$ and $d_i(\lambda 1-k_i)$ are invertible, then $(\lambda1-k_i)$ is invertible, so $\lambda \not\in \sigma(k_i)$. Therefore,  $\lbrace i \in \mathbb{N}: \lambda \not\in \sigma(k_i)\rbrace \in \omega$ and, as $\omega$ is a filter, $\lbrace i \in \mathbb{N}: \lambda \in \sigma(k_i)\rbrace \not\in \omega$. 
\end{proof}


An element $a$ in a $C^*$-algebra $A$ is called \emph{full} if it generates all of $A$ as a closed, two-sided ideal. A $^*$-homomorphism $\pi:A\rightarrow{B}$ is called \textit{full} if $\pi(a)$ is full in $B$ for every non-zero $a\in A$. The next proposition allows us to test for fullness in the setting of Theorem \ref{MainTheorem}.

\begin{proposition}\label{prop:Fullness}
    Consider a sequence of unital extensions $(0 \rightarrow J_i \rightarrow A_i \xrightarrow{\pi_i} B_i \rightarrow 0)_i$, where the $B_i$ are Kirchberg algebras and each $J_i$ is stable. Let 
    \begin{equation}
         0 \rightarrow \prod_{i\rightarrow\omega}J_i \rightarrow \prod_{i\rightarrow\omega}A_i \xrightarrow{\pi_\omega} \prod_{i\rightarrow\omega}B_i \rightarrow 0
    \end{equation}
    be the induced extension of the corresponding ultraproducts. 
    \begin{enumerate}[(i)]
        \item An element $a \in \prod_{i\rightarrow\omega}A_i$ is full if and only if $\pi_\omega(a) \neq 0$.
        \label{lem:Fullness}
        \item Moreover, if $D$ is a simple unital $C^*$-algebra, $\phi:D \rightarrow \prod_{i\rightarrow\omega}A_i$ is a c.p.c.\ order zero map and $\Phi:C_0(0,1] \otimes D \rightarrow \prod_{i\rightarrow\omega}A_i$ is the corresponding $^*$-homomorphism out of the cone, then the following are equivalent:
\begin{enumerate}[(a)]
    \item The element $\pi_\omega(\phi(1))$ has spectrum $[0,1]$. \label{item:fullspectrum}
    \item The map $\pi_\omega \circ \Phi$ is injective. \label{item:injective}
    \item The map $\Phi$ is full. \label{item:full}
\end{enumerate}
    \end{enumerate}
    
\end{proposition}
\begin{proof}
   (i): The condition $\pi_\omega(a) \neq 0$ is clearly necessary for fullness of $a$, we prove that it is sufficient. Suppose that $a \in \prod_{i\rightarrow\omega}A_i$ satisfies $\pi_\omega(a) \neq 0$, and let $I$ be the ideal generated by $a$. As $0 \neq \pi_\omega(a)^*\pi_\omega(a) \in I$ and $\pi_\omega(a^*a) = \pi_\omega(a)^*\pi_\omega(a)$, we may assume without loss of generality that $a$ is positive.
    
    Since each $B_i$ is simple and purely infinite, so is the ultraproduct $\prod_{i\rightarrow\omega}B_i$ by \cite[Remark~2.4]{Kir06}. Hence, there exists $b \in \prod_{i\rightarrow\omega}B_i$ such that $b^*\pi_\omega(a)b = 1$. 
    Let $x \in \prod_{i\rightarrow\omega}A_i$ be a lift of $b$. Then $\pi_\omega(x^*ax) = 1$. Hence, $x^*ax = (1_{A_i}-d_i)_i$ for some $d_i \in J_i$, which may be taken to be self-adjoint. 
    
    As the $J_i$ are stable, we can write $J_i = \K \otimes C_i$ for some $C^*$-algebras $C_i$. Perturbing each $d_i$ in such a way that the equivalence class of $(d_i)_i$ is unaffected, we may assume that $d_i \in M_{n_i} \otimes C_i$ for some $n_i \in \N$. Writing $p_{n_i}$ for the unit in $M_{n_i}$, we may make a further perturbation and assume there exist positive contractions $c_i \in C_i$ such that $(p_{n_i} \otimes c_i)d_i = d_i$ (these come from approximate units in each $C_i$). 
    
    Let $v_i \in \K$ be a partial isometry with $v_iv_i^* = p_{n_i}$ and $v_i^*v_i = p_{2n_i} - p_{n_i}$. Set $e = (p_{n_i} \otimes c_i)_i$ and $v = (v_i \otimes c_i^{1/2})_i$. Note that $(1-e)x^*ax=1-e$. Then  
    \begin{align}
        1 &= 1-e + e - 0\nonumber\\
        & = 1-e + vv^* - vev^*\nonumber\\
       &= 1-e + v(1-e)v^*\nonumber\\
       &= (1-e)x^*ax + v((1-e)x^*ax)v^*\in I.
    \end{align}
    Therefore, $I = \prod_{i\rightarrow\omega}A_i$, i.e.\ $a$ is full.
    
    (ii):  $(a) \Rightarrow(b)$: Let $I = \ker(\pi_\omega \circ \Phi)$. Assume for a contradiction that $I$ is non-zero. As $D$ is simple, 
    $I$ must contain some element of the form $f \otimes 1$ for some non-zero $f \in C_0(0,1]$, which we may further assume to be a positive contraction. It follows that $0 = \pi_\omega(\Phi(f \otimes 1)) = \pi_\omega(f(\phi(1)))$. This contradicts the fact that $\pi_\omega(\phi(1))$ has spectrum $[0,1]$. Therefore $I = 0$, and so $\pi_\omega\circ\Phi$ is injective. 
    
    $(b) \Rightarrow(c)$: Let $x \in C_0(0,1] \otimes D$ be non-zero. As $\pi_\omega \circ \Phi$ is assumed to be injective, we have $\pi_\omega(\Phi(x)) \neq 0$. So $\Phi(x)$ is full by part (i).
    
   $(c) \Rightarrow(a)$: Let $f \in C_0(0,1]$ be positive and of norm one. Then $\Phi(f \otimes 1)$ is full in $\prod_{i\rightarrow\omega}A_i$, so $\pi_\omega(\Phi(f \otimes 1)) = \pi_\omega(f(\phi(1)))$ is non-zero. Since $f$ was arbitrary, it follows that $\pi_\omega(\phi(1))$ has spectrum $[0,1]$.
\end{proof}

With this in place, we now give the existence result needed for our theorem.

\begin{proposition}\label{thm:Existence}
	Let $(0 \rightarrow J_i \rightarrow A_i \xrightarrow{\pi_i} B \rightarrow 0)_i$ be a sequence of unital extensions, where $B$ is a fixed unital Kirchberg algebra and each $J_i$ is stable. There exist c.p.c.\ maps $\phi_i:B \rightarrow A_i$, which factor through matrix algebras 
\begin{equation}
\begin{tikzpicture}[node distance = 4cm, auto, baseline=(current  bounding  box.center)]
\node (B) {$B$};
\node(temp) [node distance = 2cm, right of = B] {};
\node (prod)[right of = B] {$A_i$};
\node (CB)[node distance = 2cm, below of = temp] {$M_{n_i}(\mathbb{C})$};

\draw[->] (B) to node[swap] {$\eta_i$} (CB);
\draw[->] (CB) to node[swap] {$\xi_i$} (prod);
\draw[->] (B) to node {$\phi_i$} (prod);
\end{tikzpicture}
\end{equation} 
with $\eta_i$ c.p.c.\ and $\xi_i$ c.p.c.\ order zero, such that the induced map $\phi:B \rightarrow \prod_{i \rightarrow \omega} A_i$ is c.p.c.\ order zero and the corresponding $^*$-homomor\-phism $\Phi:C(0,1] \otimes B \rightarrow \prod_{i \rightarrow \omega} A_i$ is full.
\end{proposition}
\begin{proof}
Since $C_0(0,1] \otimes B$ is quasidiagonal by \cite{V}, there exists a sequence of  c.p.c.\ maps \ $\psi_i: C_0(0,1] \otimes B \rightarrow M_{n_i}(\mathbb{C})$ that is approximately multiplicative and approximately isometric.
Define $\eta_i:B \rightarrow M_{n_i}(\mathbb{C})$ by $\eta_i(b) = \psi_i(\operatorname{id}_{(0,1]} \otimes \, b)$. Since the sequence of maps $(\psi_{i})_i$ is approximately multiplicative, the sequence of maps $(\eta_i)_i$ is approximately order zero.  

Let $k \in \O_\infty$ be a positive contraction with spectrum $[0,1]$ and let $\alpha_i:M_{n_i}(\mathbb{C}) \rightarrow \O_\infty$ be a (non-unital) injective $^*$-homomorphism.\
Define a c.p.c.\ order zero map $\hat{\xi}_i:M_{n_i}(\mathbb{C}) \rightarrow B \otimes \O_\infty \otimes \O_\infty$ by $\hat{\xi}_i(x) = 1_B \otimes \alpha_i(x) \otimes k$. As $B$ is a Kirchberg algebra, it is $\O_{\infty}$-stable by \cite{0infinitystability}, and so we may identify the codomain of $\hat{\xi}_i$ with  $B$. The resulting c.p.c.\ order zero map $\tilde{\xi}_i:M_{n_i}(\mathbb{C}) \rightarrow B$  has a c.p.c.\ order zero lift $\xi_i:M_{n_i}(\mathbb{C}) \rightarrow A_i$, i.e. $\pi_i\circ\xi_i=\tilde{\xi}_i$,  by \cite[Proposition 1.2.4]{Wi09}.\footnote{\cite[Proposition 1.2.4]{Wi09} is a rephrasing to the setting of order zero maps of Loring's result (\cite[Theorem 10.2.1]{LO}) that $C_0(0,1] \otimes M_{n}(\mathbb{C})$ is projective.} 

Set $\phi_i$ to be the composition $\xi_i \circ \eta_i$. Since the sequence of maps $(\eta_i)_i$ is approximately order zero, and each map $\xi_i$ is order zero, the map  $\phi:B \rightarrow \prod_{i \rightarrow \omega} A_i$ induced by $(\phi_i)_i$ is c.p.c.\ order zero.

Let $\Phi:C(0,1] \otimes B \rightarrow \prod_{i \rightarrow \omega} A_i$ be the $^*$-homomorphism from the cone over $B$ corresponding to the c.p.c.\ order zero map $\phi$. Using that the spectrum of an elementary tensor is the product of the spectra of its factors (see \cite{BP66}), we compute that 
\begin{align}
\sigma(\pi_i(\phi_{i}(1_B)))&=\sigma(\hat{\xi}_i(\eta_i(1_B)))\nonumber\\
&=\sigma(1_B \otimes \alpha_i\psi_{i}(\operatorname{id}_{(0,1]} \otimes 1_B) \otimes k)\nonumber\\
& \supseteq [0, \|\psi_{i}(\operatorname{id}_{(0,1]} \otimes 1_B)\|]
\end{align}
as $\alpha_i$ is isometric and $k$ is a positive contraction of full spectrum. Since $(\psi_i)_i$ are approximately isometric, we have $\lim_{i\to\omega} \|\psi_{i}(\operatorname{id}_{(0,1]} \otimes 1_B)\| = 1$. By Proposition \ref{remark},  $\sigma(\pi_{\omega}(\phi(1_B))) \supseteq [0,1)$. Hence,  $\sigma(\pi_{\omega}(\phi(1_B))) = [0,1]$. Therefore, $\Phi$ is full by Proposition~\ref{prop:Fullness}. 
\end{proof}

Our final ingredient is a complementary uniqueness result showing that, after amplifying the codomain by a matrix algebra, the $^\ast$-homomorphisms constructed in the previous proposition are unique up to unitary equivalence. To achieve this we use a special case of Gabe's new approach to $\mathcal O_\infty$-stable classification (\cite{Ga19}).  For ease of notation, given a $^*$-homomorphism $\phi:A\rightarrow B$, we write $\phi\oplus 0$ for the $^*$-homomorphism $A\rightarrow M_2(B)$ given by
\begin{equation}
	a \mapsto \begin{pmatrix}
\phi(a) & 0\\
0  & 0
\end{pmatrix}.
\end{equation}

\begin{proposition}\label{thm:OinfClassification}
 Let $(A_i)_i$ be a sequence of separable unital $C^*$-algebras and $B$ be a Kirchberg algebra. Suppose $\Phi, \Psi:C_0(0,1] \otimes B \rightarrow \prod_{i\to\omega}A_i$ are full $^*$-homomorphisms. Then there exists a unitary $u \in \prod_{i\to\omega} M_{2}(A_i)$ such that
  \begin{align}
        \Psi(x)\oplus 0 &= u(\Phi(x)\oplus 0)u^*,\quad x\in C_0(0,1]\otimes B.
         \end{align} 

\end{proposition}
\begin{proof}
 Since $B$ is nuclear and $\mathcal{O}_{\infty}$-stable, it is automatic that $\Phi$ and $\Psi$ are nuclear and strongly $\mathcal{O}_{\infty}$-stable. Indeed, $^*$-homomorphisms with $\mathcal{O}_\infty$-stable domains are strongly $\mathcal{O}_\infty$-stable by  \cite[Proposition~3.18]{Gabe}.\footnote{Since the concept of a strongly $\O_\infty$-stable map plays no concrete role in the remainder of this article, we shall refrain from repeating the definitions here and refer the reader to Gabe's article \cite{Ga19} instead.} Moreover, by the contractibility of $ C_{0}(0,1]\otimes B$ the $^*$-homomorphisms $\Phi$ and $\Psi$ are homotopic to zero, so $KK( \Phi)=KK(\Psi)=0$.\footnote{ Similarly, we refrain from defining Kasparov's $KK$-groups as the only properties we require are that $KK$ is homotopy invariant and that $KK(0)=0$.}  Therefore, by \cite[Theorem B]{Ga19}, $\Phi$ and $\Psi$ are asymptotically Murray-von Neumann equivalent in the sense of \cite[Definition 2.1]{Ga19}. 
 
Using \cite[Proposition 3.10]{Gabe}, we deduce that $\Phi \oplus 0$ and $\Psi \oplus 0$ are asymptotically unitarily equivalent and hence approximately unitarily equivalent. As $B$ is separable, a standard reindexing argument or application of Kirchberg's $\varepsilon$-test (see \cite[Lemma~A.1]{Kir06}) enables one to pass from approximate unitary equivalence to unitary equivalence, i.e.\ there exists a unitary $u \in \prod_{i\rightarrow \omega} M_2(A_i) $ such that for all $x \in C_0(0,1]\otimes B$, $ \Psi(x)\oplus 0 = u(\Phi(x)\oplus 0)u^*$, as desired.   
\end{proof}
\noindent We are now in a position to prove Theorem \ref{MainTheorem}.
\begin{proof}[Proof of Theorem \ref{MainTheorem}]
Let $\pi:A \rightarrow B$ denote the quotient map of the extension. Fix a unital c.p.\ splitting $\mu:B \rightarrow A$, which can be obtained from the Choi-Effros lifting theorem \cite{Choi-Effros}. As $B \cong B \otimes \O_\infty \otimes \O_\infty \otimes \cdots$, we can fix an approximately central sequence of positive contractions $(k_i)_i$ in $B$ such that $\sigma(k_i) = [0,1]$ for all $i \in \N$.

We will start by constructing a quasicentral approximate unit $(h_i)_i$ for the extension with some additional properties. Firstly, we shall need that each $h_i$ has finite spectrum and $h_{i+1}h_i = h_i$ for all $i \in \N$. Secondly, we need to ensure that the sequences of c.p.c.\ maps $\alpha_i:A \rightarrow A$, $\theta^{(0)}_i, \theta^{(1)}_i:B \rightarrow A$ given by
 \begin{align}
     \alpha_i(a)&= h_i^{\frac{1}{2}} a h_i^{\frac{1}{2}} \nonumber\\
    \theta^{(0)}_i(b)&= (h_{i+1}-h_i)^{\frac{1}{2}}\mu(b)(h_{i+1}-h_i)^{\frac{1}{2}}\nonumber\\
    &\quad\quad+(1-h_{i+1})^{\frac{1}{2}}\mu(k_i^{\frac{1}{2}}b k_i^{\frac{1}{2}})(1-h_{i+1})^{\frac{1}{2}}, \notag \\
    \theta^{(1)}_i(b)&=(1-h_{i+1})^{\frac{1}{2}}\mu((1-k_i)^{\frac{1}{2}}b (1-k_i)^{\frac{1}{2}})(1-h_{i+1})^{\frac{1}{2}},
 \end{align}
for $a\in A$ and $b\in B$, induce order zero maps $\alpha:A \rightarrow A_\omega$, $\theta^{(0)},\theta^{(1)}:B\rightarrow A_\omega$, and that the sum $\alpha + \theta^{(0)} \circ \pi + \theta^{(1)} \circ \pi$ coincides with the canonical embedding $A \rightarrow A_\omega$ as constant sequences. 
  
We now perform the construction of $(h_i)_i$. Write $J$ as an increasing union of finite-dimensional $C^*$-algebras $(F_n)_n$, and let $e_n$ be the unit of $F_n$. Then $(e_n)_n$ will be an approximate unit for $J$ with $e_{n+1}e_n = e_n$ for all $n \in \N$. Therefore, there exists a quasicentral approximate unit $(h_i)_i$ in the convex hull of the sequence $(e_n)_n$ such that $h_{i+1}h_i = h_i$ for all $i \in \N$ (see \cite[II.4.3.2]{Bl06}).  By construction, each $h_i$ lies in a finite-dimensional subalgebra, so has finite spectrum. Since $B$ is separable and $(k_i)_i$ has already been fixed, replacing $(h_i)_{i}$ with a subsequence, we may further assume that
\begin{align}\label{eqn:subsequence3}
    \|h_i\mu(bk_i)-\mu(bk_i)h_i\| &\rightarrow 0, \nonumber \\
    \|h_{i+1}\mu(bk_i)-\mu(bk_i)h_{i+1}\| &\rightarrow 0,\nonumber \\
    \|(1-h_{i+1})(\mu(b)\mu(b'k_i)-\mu(bb'k_i))\| &\rightarrow 0,\nonumber \\
    \|(1-h_{i+1})(\mu(bk_i)\mu(b'k_i)-\mu(bk_i b'k_i))\| &\rightarrow 0 
 \end{align}
for all $b, b' \in B$.

With this choice of $(h_i)_i$ we can now verify that the maps $\alpha$, $\theta^{(0)}$ and $\theta^{(1)}$ are order zero. For $\alpha$ this is immediate from quasicentrality, but for $\theta^{(0)}$ and $\theta^{(1)}$ we shall need \eqref{eqn:subsequence3}. Set $h=(h_i)_{i} \in A_\omega \cap A'$,\ $\tilde{h}=(h_{i+1})_{i} \in A_\omega \cap A'$, $k=(k_i)_{i} \in B_\omega \cap B'$, and write $\mu_\omega:B_\omega \rightarrow A_\omega$ for the c.p.c.\ map induced by $\mu$. Let $b,b'\in B$ be positive elements such that $bb'=0$. Then we have \begin{align}
   \quad\quad \theta^{(0)}(b)\theta^{(0)}(b')&=(\tilde{h}-h)^{\frac{1}{2}}\mu_\omega(b)(\tilde{h}-h)\mu_\omega(b')(\tilde{h}-h)^{\frac{1}{2}}\notag\\
    &\quad +(\tilde{h}-h)^{\frac{1}{2}}\mu_\omega(b)(\tilde{h}-h)^{\frac{1}{2}}(1-\tilde{h})^{\frac{1}{2}}\mu_\omega(b'k)(1-\tilde{h})^{\frac{1}{2}}\notag\\
    &\quad +(1-\tilde{h})^{\frac{1}{2}}\mu_\omega(bk) (1-\tilde{h})^{\frac{1}{2}}(\tilde{h}-h)^{\frac{1}{2}}\mu_\omega(b')(\tilde{h}-h)^{\frac{1}{2}}\nonumber\\
    &\quad +(1-\tilde{h})^{\frac{1}{2}}\mu_\omega(bk)(1-\tilde{h})\mu_\omega(b'k)(1-\tilde{h})^{\frac{1}{2}}.\label{orthogonal}
\end{align}
Now as $(h_i)_i$ is quasicentral for $A$ and is chosen so the properties (\ref{eqn:subsequence3}) hold, all the terms in the right hand side of (\ref{orthogonal}) are zero. Applying a similar argument to $\theta^{(1)}$ shows that it too is order zero.  Moreover,  for all $a \in A$, we have
\begin{align}\label{embed}
    (\alpha + \theta^{(0)} \circ \pi + \theta^{(1)} \circ \pi)(a)
    &= h a + (\tilde{h} - h) \mu_\omega(\pi(a)) \quad\quad\quad\quad \nonumber\\
    &\quad + (1-\tilde{h}) \mu_\omega ( \pi(a) k ) \nonumber\\
    &\quad + (1-\tilde{h}) \mu_\omega ( \pi(a)(1-k) )\nonumber\\
    &=a
\end{align}
using \eqref{eqn:subsequence3} and that $k \in B_\omega \cap B'$ and $h,\tilde{h} \in A_\omega \cap A'$. 

Since $h_{i+1}h_i = h_i$, the maps $\alpha_i$ and $\theta^{(1)}_i$ have orthogonal ranges. To ensure that this orthogonality is respected by subsequent constructions, we introduce hereditary subalgebras 
\begin{align}
          A_i &= \overline{(1-h_{i+1})A(1-h_{i+1})} \subseteq A,\nonumber\\
          J_i &= \overline{(1-h_{i+1})J(1-h_{i+1})} \subseteq J,\nonumber\\
          D_i &= \overline{h_iAh_i} \subseteq J. 
\end{align}
By restricting the codomains, we may view $\alpha_i$ as a map  $A \rightarrow D_i$ and $\theta^{(1)}_i$ as a map $B \rightarrow A_i$, and we note that $A_i$ and $D_i$ are orthogonal. We also have an induced short exact sequence $0\rightarrow J_i \rightarrow A_i \rightarrow B\rightarrow 0$, where the quotient map is the restriction of $\pi$. Since $h_{i+1}$ has finite spectrum, $A_i$ is unital. Moreover, $J_i$ is stable by \cite[Corollary~2.3(iii)]{Stable}.

Let $\Theta^{(0)}:C_0(0,1] \otimes B \rightarrow A_\omega$ be the $^*$-homomorphisms corresponding to the c.p.c.\ order zero map $\theta^{(0)}$. By Proposition \ref{prop:Fullness}, $\Theta^{(0)}$ is full if $\pi_\omega(\theta^{(0)}(1))$ has spectrum $[0,1]$, where $\pi_\omega:A_\omega \rightarrow B_\omega$ is the map on ultraproducts induced by $\pi$. As $h,\tilde{h} \in J_\omega$, we have
\begin{equation}\pi_\omega(\theta^{(0)}(1))=\pi_\omega(\tilde{h}-h)+\pi_\omega(1-\tilde{h})\pi_\omega(\mu_\omega(k))=k
\end{equation}
which has spectrum $[0,1]$ by Proposition \ref{remark}. Therefore, $\Theta^{(0)}$ is full. By the same argument applied to the extensions $0\rightarrow J_i \rightarrow A_i \rightarrow B\rightarrow 0$, the $^*$-homomorphism $\Theta^{(1)}:C_0(0,1] \otimes B \rightarrow \prod_{i\rightarrow \omega} A_i$ induced by the c.p.c.\ order zero map $\theta^{(1)}:B \rightarrow \prod_{i\rightarrow \omega} A_i$ is full because $\pi_\omega(\theta^{(1)}(1)) = 1-k$  has spectrum $[0,1]$.

Let $(\phi^{(0)}_i:B \rightarrow A)_i$ be the sequence of c.p.c.\ maps provided by Proposition \ref{thm:Existence} corresponding to the constant sequence of extensions $0\rightarrow J \rightarrow A \rightarrow B\rightarrow 0$, and let $(\phi^{(1)}_i:B \rightarrow A_i)_i$ be the sequence of c.p.c.\ maps corresponding to the sequence of extensions $(0\rightarrow J_i \rightarrow A_i \rightarrow B\rightarrow 0)_i$. By construction, $\phi^{(j)}_i$ factor through a matrix algebra $M^{(j)}_i$ as $\xi_i^{(j)}\circ \eta_i^{(j)}$ with $\eta_i^{(j)}$ c.p.c.\ and $\xi_i^{(j)}$ c.p.c.\ order zero ($j=1,2$), the induced maps $\phi^{(0)}_i:B \rightarrow A_\omega$ and $\phi^{(1)}_i:B \rightarrow \prod_{i\rightarrow \omega} A_i$ are order zero, and the corresponding $^*$-homomorphisms $\Phi^{(0)}, \Phi^{(1)}$ from the cone over $B$ are full.

Proposition \ref{thm:OinfClassification} establishes that there exist unitaries $u\in M_2(A)_\omega$ and $v\in \prod_{i\rightarrow \omega} M_2(A_i)$ such that
\begin{align}
\operatorname{Ad}(u)\circ (\Phi^{(0)}\oplus 0)&=\Theta^{(0)}\oplus 0,\nonumber\\
\operatorname{Ad}(v)\circ (\Phi^{(1)}\oplus 0)&=\Theta^{(1)}\oplus 0.
\end{align}
In particular,
\begin{align}
    \operatorname{Ad}(u)\circ(\phi^{(0)}\oplus 0)&=\theta^{(0)}\oplus 0, \nonumber\\
    \operatorname{Ad}(v)\circ (\phi^{(1)}\oplus 0)&=\theta^{(1)}\oplus 0. \label{gamma1}
\end{align}
By \cite[Lemma~6.2.4]{Rordam2002}, $u$ and $v$ lift to sequences of unitaries $(u_i)_i\in M_2(A)$ and $(v_i)_i \in M_2(A_i)$ respectively.

Since $A$ is separable and each $D_i$ is approximately finite dimensional (by virtue of being a hereditary subalgebra of $J$; see \cite[Theorem 3.1]{El76b}), there exist finite-dimensional subalgebras $G_i \subseteq D_i$ and conditional expectations $E_i:D_i \rightarrow G_i$ such that $\|E_i(\alpha_i(a))-\alpha_i(a)\| \rightarrow 0$ for all $a \in A$. We now define completely positive approximations
\begin{equation}
\begin{tikzpicture}[node distance = 4cm, auto, baseline=(current  bounding  box.center)]
\node (B) {$A$};
\node(temp) [node distance = 2cm, right of = B] {};
\node (prod)[right of = B] {$M_2(A)$};
\node (CB)[node distance = 2cm, below of = temp] {$F^{(0)}_i \oplus F^{(1)}_i$,};

\draw[->] (B) to node[swap] {$\psi_i$} (CB);
\draw[->] (CB) to node[swap] {$\chi_i$} (prod);
\draw[->] (B) to node {$\operatorname{id} \oplus 0$} (prod);
\end{tikzpicture}
\end{equation} 
where 
\begin{align}
	F^{(0)}_i &= M^{(0)}_i,\nonumber\\
	F^{(1)}_i &= M^{(1)}_i \oplus G_i,\nonumber\\
	\psi_i(a) &= (\eta^{(0)}_i(\pi(a)), \eta^{(1)}_i(\pi(a)), E_i(\alpha_i(a))),\nonumber\\
	\chi_i(x,y,z) &= \operatorname{Ad}(u_i)(\xi_i^{(0)}(x) \oplus 0) + \operatorname{Ad}(v_i)(\xi_i^{(1)}(x) \oplus 0) + z \oplus 0. 
\end{align}
It follows from (\ref{embed}) and (\ref{gamma1}) that $\chi_i(\psi_i(a)) \rightarrow a$ along the ultrafilter $\omega$. The restriction of $\chi_i$ to $F^{(0)}_i$ is order zero because $\xi_i^{(0)}$ is order zero. Moreover, the restriction of $\chi_i$ to $F^{(1)}_i$ is order zero because $\xi_i^{(1)}:M_i^{(1)} \rightarrow A_i$ is order zero, $v_i \in M_2(A_i)$, and $D_i$ is orthogonal to $A_i$. Therefore, $\operatorname{id}_A \oplus \, 0$ has nuclear dimension at most $1$.  

As $A$ is a hereditary subalgebra of $M_2(A)$, \cite[Proposition 1.6]{BGSW19} implies that $\mathrm{id}_A$, as the co-restriction of $\operatorname{id}_A\oplus\ 0$ to $A$, has nuclear dimension at most $1$. Hence, $\dim_{\mathrm{nuc}}(A)\leq 1$. On the other hand, \cite[Proposition 2.9]{WZ10}\ implies that the nuclear dimension of $A$ is bounded below by the nuclear dimension of $B$. As we assume that $B$ is a Kirchberg algebra, it has nuclear dimension $1$ by \cite[Theorem~G]{BBSTWW}. Therefore, we conclude $\dim_{\mathrm{nuc}}(A)=1$.\end{proof}

\begin{corollary}
    The Cuntz--Toeplitz algebras $\mathcal{T}_n$ have nuclear dimension one.
\end{corollary}

\addlines[2]
{\Small
\noindent \textsc{PE: Peterhouse, Cambridge, CB2 1RD, UK.}

\emph{Email address}: \texttt{pe265@cam.ac.uk}

\smallskip

\noindent \textsc{EG, SK, DN, MV, DW, CW, CB, MW, JZ:  School of Mathematics and Statistics, University of Glasgow, Glasgow, G12 8QQ, UK.}

\emph{Email address}: \texttt{2261998G@student.gla.ac.uk}

\emph{Email address}: \texttt{2263206K@student.gla.ac.uk}

\emph{Email address}: \texttt{2197941N@student.gla.ac.uk}

\emph{Email address}: \texttt{2197055V@student.gla.ac.uk}

\emph{Email address}: \texttt{2173807W@student.gla.ac.uk}

\emph{Email address}: \texttt{2253382W@student.gla.ac.uk}

\emph{Email address}: \texttt{christian.bonicke@glasgow.ac.uk}

\emph{Email address}: \texttt{mike.whittaker@glasgow.ac.uk}

\emph{Email address}: \texttt{joachim.zacharias@glasgow.ac.uk}

\smallskip
\noindent\textsc{SE, SGP, SW: Mathematical Institute, University of Oxford, Oxford, OX2 6GG, UK.}

\emph{Email address}: \texttt{samuel.evington@maths.ox.ac.uk}

\emph{Email address}: \texttt{sergio.gironpacheco@maths.ox.ac.uk}

\emph{Email address}: \texttt{stuart.white@maths.ox.ac.uk}

\smallskip
\noindent\textsc{MF:  Institute of Mathematics, Czech Academy of Sciences, \v{Z}itn\'{a} 25, 115 67 Prague 1, Czech Republic}

\emph{Email address}: \texttt{mforough86@gmail.com}

\smallskip
\noindent\textsc{NS: School of Mathematics and Applied Statistics, University of Wollongong, NSW 2522, Australia.}

\emph{Email address}: \texttt{ns092@uowmail.edu.au}
}
\end{document}